\documentclass[10pt, article]{amsart}

\usepackage{tikz}
\usetikzlibrary{calc}
\usepackage{ae} 
\usepackage[T1]{fontenc}
\usepackage[cp1250]{inputenc}
\usepackage{amsmath}
\usepackage{amssymb, amsfonts,amscd,verbatim}

\usepackage[normalem]{ulem}
\usepackage{hyperref}
\usepackage{indentfirst}
\usepackage{latexsym}
\input xy
\xyoption{all}

\usepackage{amsmath}    

\theoremstyle{plain}
\newtheorem{Pocz}{Poczatek}[section]
\newtheorem{Proposition}[Pocz]{Proposition}
\newtheorem{Theorem}[Pocz]{Theorem}
\newtheorem{Corollary}[Pocz]{Corollary}

\newtheorem{Lemma}[Pocz]{Lemma}

\newtheorem{Problem}[Pocz]{Problem}

\theoremstyle{definition}
\newtheorem{Definition}[Pocz]{Definition}

\theoremstyle{remark}
\newtheorem{Remark}[Pocz]{Remark}

\DeclareMathOperator*{\dist}{dist}
\DeclareMathOperator*{\diam}{diam}

\def\asdim{\mathrm{asdim}}
\def\card{\mathrm{card}}

\def\dim{\mathrm{dim}}
\def\diam{\mathrm{diam}}

\numberwithin{equation}{section}
\newcommand{\X}{\mathcal{X}}
\newcommand{\N}{{\mathbb{N}}}
\def\lsdim{\mathrm{LsExtDim}}

\title[
Coarse amenability and discreteness
]%
  {Coarse amenability and discreteness}

\author{Jerzy ~Dydak}
\address{University of Tennessee, Knoxville, USA}
\email{jdydak@utk.edu}

\date{ \today
}


\begin{document}
\maketitle

\tableofcontents

\begin{abstract}
This paper is devoted to dualization of paracompactness to the coarse category via the concept of $R$-disjointness. Property A of  G.Yu can be seen as a coarse variant of amenability via partitions of unity and leads to a dualization of paracompactness via partitions of unity. On the other hand, finite decomposition complexity of Erik Guentner,  Romain Tessera, and  G.Yu and straight finite decomposition complexity of Dranishnikov-Zarichnyi employ $R$-disjointness as the main concept. We generalize both concepts to that of countable asymptotic dimension and our main result shows that it is a subclass of of spaces with Property A. In addition, it gives a necessary and sufficient condition for spaces of countable asymptotic dimension to be of finite asymptotic dimension.
\end{abstract}

\section{Introduction}

\textbf{Property A} of G.Yu \cite{Yu00} was introduced in the context of the Novikov Conjecture.
It is a large scale variant of amenability. See \cite{Willett} for a survey of results on Property A. 
Subsequently, it was generalized to the concept of \textbf{exact spaces} by Dadarlat and Guentner \cite{DaGu}. In \cite{CDV2} exact spaces were narrowed down to \textbf{large scale paracompact spaces}
and \cite{CDV4} (see also \cite{CDV3}) contains an analysis of interrelationships between various concepts.

As explained in \cite{CDV4} all the above concept can be unified using existence (for each $\epsilon > 0$) of
$(\epsilon,\epsilon)$-Lipschitz (see \ref{LipschitzFunctionsDef}) partitions of unity
$f:X\to \Delta(S)$ (see \ref{PUDef}) that are cobounded (see \ref{CoboundedDef}).
Property A corresponds to $f$ being a barycentric partition of unity (see \ref{PUDef}),
exact spaces correspond to arbitrary partitions of unity, and large scale paracompact spaces
correspond to the case of $f$ having Lebesgue number at least $\frac{1}{\epsilon}$ (see \ref{PUDef}).

One may summarize that the three concepts (Property A, exact spaces, large scale paracompact spaces) deal with dualizing paracompactness via partitions of unity. In \cite{CDV4} the concept
of \textbf{Strong Property A} was introduced as a way of dualizing paracompactness via covers.

This paper is devoted to developing large scale paracompactness from the point of view of discreteness. More precisely, it deals with dualizing the following two classical results of general topology:

\begin{Theorem}[Michael-Nagami \cite{En}] \label{Michael-NagamiThm}
A Hausdorff space $X$ is paracompact if and only if it is weakly paracompact and collectionwise normal.
\end{Theorem}

\begin{Theorem}\label{DiscreteCharOfParacompactness}(see Theorem 5.1.12 in \cite{En}, p.303)
A regular space 
$X$ is paracompact if and only if every open cover has a $\sigma$-discrete open refinement.
\end{Theorem}

Since topological discreetness naturally dualizes to $R$-disjointness (see \ref{R-disjointDef}),
one arises at the following question:

\begin{Problem}\label{RDisjointnessProblem}
Characterize metric spaces $X$ such that for each $R > 0$ there exists $M > 0$ and a finite sequence
of $R$-disjoint families $\mathcal{U}_n$, $i\leq k$, such that $X=\bigcup\limits_{i=1}^k \mathcal{U}_n$
and diameters of elements of each $\mathcal{U}_n$ are at most $M$.
\end{Problem}

It turns out special cases of \ref{RDisjointnessProblem} were considered in the past.
The most restrictive property expressed in terms of $R$-disjointness is the following.

\begin{Definition} [Dranishnikov \cite{Dran AsyTop}]
A metric space $X$ has \textbf{asymptotic property C} if for every sequence $R_1 < R_2 < \ldots$ there exists $n\in \mathbb{N}$
such that $X$ is the union of $R_i$-disjoint families $\mathcal{U}_i$, $1\leq i\leq n$,
that are uniformly bounded.
\end{Definition}

Subsequently, E.Guentner,  R.Tessera, and G.Yu introduced the concept of \textbf{finite decomposition complexity} (see \cite{YuNo})
which was weakened as follows:

\begin{Definition}
[Dranishnikov-Zarichnyi \cite{DranZari}]
 $X$ is of \textbf{straight finite decomposition complexity}  if 
for any increasing sequence of positive real numbers $R_1 < R_2 < \ldots$ there a sequence
$\mathcal{V}_i$, $i \leq n$, of families of subsets of $X$ such that the following conditions are satisfied:
\begin{itemize}
\item[1.] $\mathcal{V}_1=\{X\}$,
\item[2.] each element $U\in \mathcal{V}_i$, $i < n$, can be expressed as a union of at most $2$ families from $\mathcal{V}_{i+1}$ that are $R_i$-disjoint,
\item[3.] $\mathcal{V}_n$ is uniformly bounded.
\end{itemize}
\end{Definition}

It turns out straight finite decomposition complexity is a variant of coarse amenability:
\begin{Theorem}
[Dranishnikov-Zarichnyi \cite{DranZari}] \label{DranZariResult}
 Every space of straight finite decomposition complexity has Property A.
\end{Theorem}

Our view is that straight finite decomposition complexity is a special case of countable asymptotic dimension (see \ref{CountableAsdim}). Namely, it corresponds to the fact that, in topology, one can define spaces of countable covering dimension as either countable unions of zero-dimensional spaces or as countable unions of spaces of finite dimension. Our main result, Theorem \ref{MainResult}, states that spaces $X$ of countable asymptotic dimension are actually of finite asymptotic dimension provided some finite skeleton of $\Delta(X)$ is a large scale absolute extensor of $X$. It generalizes \ref{DranZariResult} as well.

The author is grateful to the referee for valuable comments and suggestions that improved the exposition of the paper.

\section{Basic concepts}

In this section we recall basic concepts used in the paper.

\begin{Definition}
The \textbf{cardinality} of a set $S$ is denoted by $\card(S)$.
\end{Definition}

\begin{Definition}\label{R-disjointDef}
Given $R > 0$, a family $\{U_s\}_{s\in S}$ of subsets of a metric space $X$ is called $R$-\textbf{disjoint}
if $d(x,y) > R$ whenever $x\in U_s$, $y\in U_t$, and $s\ne t$.
\end{Definition}

\begin{Definition}\label{UniformlyBoundedDef}
A family $\{U_s\}_{s\in S}$ of subsets of a metric space $X$ is called \textbf{uniformly bounded}
if there is $M > 0$ such that diameters of all sets of the family are at most $M$.
\end{Definition}

\begin{Definition}\label{LebesgueNumberDef}
The \textbf{Lebesgue number} of a family $\{U_s\}_{s\in S}$ of subsets of a metric space $X$ is at least $M > 0$ if the family of $M$-balls $\{B(x,M)\}_{x\in X}$ refines $\{U_s\}_{s\in S}$.
\end{Definition}

\begin{Definition}\label{DeltaDef}
By $\Delta(S)$ we mean the subspace of $l_1(S)$ ($S$ is the set of vertices of the simplicial complex $\Delta(S)$) consisting of non-negative functions $f: S\to [0,1]$ of finite support
such that $\sum\limits_{v\in S} f(v)=1$. The {\bf star} $st(v)$ of vertex $v$ consists of all $f\in \Delta(S)$
such that $f(v) > 0$.

By $\Delta(S)^{(n)}$ we mean the \textbf{$n$-skeleton} of $\Delta(S)$.
\end{Definition}

\begin{Definition}\label{PUDef}
A (point-finite) \textbf{partition of unity} on a set $X$ is a function $f:X\to \Delta(S)$ for some $S$.
$f$ is a \textbf{barycentric partition of unity} if $f(x)(v)=f(x)(w)$ whenever $f(x)(v) > 0$ and $f(x)(w) > 0$.

The \textbf{Lebesgue number} of $f$ is synonymous with the Lebesgue number of
$\{f^{-1}(st(v))\}_{v\in S}$.
\end{Definition}

\begin{Definition}\label{CoboundedDef}
Suppose $X$ is a metric space.
A partition of unity $f:X\to \Delta(S)$ is $M$-\textbf{cobounded} if
$\diam(f^{-1}(st(v)))\leq M$ for all $v\in S$.

$f$ is \textbf{cobounded} if it is $M$-cobounded for some $M > 0$.
\end{Definition}

\begin{Definition}\label{LipschitzFunctionsDef}
A function $f:X\to Y$ is $(\lambda,C)$-Lipschitz if $d_Y(f(x),f(y))\leq \lambda\cdot d_X(x,y)+C$
for all $x,y\in X$.
\end{Definition}

\begin{Lemma}\label{BasicLipLemma}
Suppose $f:X\to \Delta(S)$ is a partition of unity and $\epsilon \ge \frac{2}{R+1}$ for some $R > 0$.
If $d(x,y) < R$ implies $d(f(x),f(y))\leq \epsilon\cdot d(x,y)+\epsilon$, then $f$ is $(\epsilon,\epsilon)$-Lipschitz.
\end{Lemma}
\begin{proof}
If $d(x,y)\ge R$, then $ \epsilon\cdot d(x,y)+\epsilon\ge \epsilon\cdot(R+1)\ge 2\ge d(f(x),f(y))$.
\end{proof}

For basic facts related to the coarse category see \cite{Roe lectures}.

\section{Large scale weak paracompactness}

A dualization of weak paracompactness was developed in \cite{CDV4} via coarsening of covers.
Using $R$-disjointness one is led to a different concept and we do not know if it is equivalent to large scale weak paracompactness (see Problems \ref{RDisjointnessProblem} and \ref{WeakParaEquivalenceProblem}).

\begin{Definition}[\cite{CDV3},  \cite{CDV4}]\label{LSWeakParaDef}
$X$ is \textbf{large scale weakly paracompact} if for each $r,s > 0$ there is a uniformly bounded cover $\mathcal{U}$ of $X$ of Lebesgue number at least $s$ such that every $r$-ball $B(x,r)$ is contained in only finitely many elements of $\mathcal{U}$.
\end{Definition}

\begin{Proposition}[\cite{CDV4}]\label{CharacterisationLSWP}
The following conditions are equivalent for each metric space $X$:\\
a. For each $r > 0$ there is a uniformly bounded cover $\mathcal{U}$ of $X$ such that every $r$-ball $B(x,r)$ intersects only finitely many elements of $\mathcal{U}$.\\
b. $X$ is large scale weakly paracompact.\\
c. For every uniformly bounded cover $\mathcal U$ of $X$ there exists uniformly bounded point-finite cover $\mathcal V$ such that $\mathcal U$ is refinement of $\mathcal V$.
\end{Proposition}

The following is a partial answer to Problem \ref{RDisjointnessProblem}:

\begin{Proposition}
If for every $r > 0$ there is a uniformly bounded cover $\mathcal{U} $ of $X$ that can be written as the union
$\bigcup\limits_{i=1}^{\infty} \mathcal{U} _i$ of $r$-disjoint families $\mathcal{U} _i$, then $X$ is large scale weakly paracompact.
\end{Proposition}
\begin{proof} Suppose $s > 0$.
Pick a uniformly bounded cover $\mathcal{U} $ of $X$ that can be written as the union
$\bigcup\limits_{i=1}^{\infty} \mathcal{U} _i$ of $2s$-disjoint families $\mathcal{U} _i$

Given $U\in \mathcal{U} _k$ define $U'=U\setminus \bigcup\limits_{i < k} \{B(V,s) | V\in \mathcal{U} _i\}$
and $U^\ast =B(U',s)$.
Since $\{U'\}_{U\in \mathcal{U}}$ is a uniformly bounded cover of $X$, $\{U^\ast\}_{U\in \mathcal{U}}$ is of Lebesgue number at least $s$ and is uniformly bounded. Given $x\in X$ choose $m\ge 1$ so that $x\in U$
for some $U\in \mathcal{U}_m$. Therefore $B(x,s)\cap V'=\emptyset$ and $x\notin V^\ast$ for all $V\in \mathcal{U}_i$, $i > m$. If we fix $k\leq m$, then there is at most one $V\in \mathcal{U}_k$ such that $x\in V^\ast$.
Thus $\{U^\ast\}_{U\in \mathcal{U}}$  is a point-finite cover of $X$. By c) of \ref{CharacterisationLSWP}, $X$ is large scale weakly paracompact.
\end{proof}

\begin{Corollary} [\cite{CDV4}]
If $X$ is separable at some scale $r > 0$ (that means there is a countable subset $S$
of $X$ with $\bigcup\limits_{x\in S}B(x,r)=X$), then $X$ is large scale weakly paracompact.
\end{Corollary}
\begin{proof}
The family $\{B(x,r)\}_{x\in S}$ is uniformly bounded and is the union of countably many $\infty$-disjoint families.
\end{proof}

\begin{Problem}\label{WeakParaEquivalenceProblem}
Suppose $X$ is large scale weakly paracompact and $r > 0$. Is there a uniformly bounded cover $\mathcal{U} $ of $X$ that can be written as the union
$\bigcup\limits_{i=1}^{\infty} \mathcal{U} _i$ of $r$-disjoint families $\mathcal{U} _i$?
\end{Problem}

\begin{Definition}[\cite{CDV3}]\label{LSFinitisticDef}
A metric space $X$ is \textbf{large scale finitistic} if for every $r > 0$
there is a uniformly bounded cover $\mathcal{U}$ of $X$ whose Lebesgue number
 is at least $r$ and there is $n(\mathcal{U})\in \mathbb{N}$ such that each $x\in X$ belongs to at most $n(\mathcal{U})$ elements of $\mathcal{U}$.
\end{Definition}

\begin{Problem}
Suppose $X$ is large scale finitistic and $r > 0$. Is there a uniformly bounded cover $\mathcal{U} $ of $X$ that can be written as the union
$\bigcup\limits_{i=1}^{m} \mathcal{U} _i$ of finitely many $r$-disjoint families $\mathcal{U} _i$?
\end{Problem}

\section{Pasting partitions of unity}

This section contains the main technical tool of the paper: pasting partitions of unity
so that the resulting partition of unity is $(\epsilon,\epsilon)$-Lipschitz and $K$-cobounded. Given a partition of unity $f:A\to \Delta(S)$, by the \textbf{carrier} of $f$ we mean the minimal subcomplex of $\Delta(S)$ containing $f(A)$.

\begin{Lemma}\label{BasicExtLemma}
Suppose the following is given:

\begin{itemize}
\item[a.] $A$ is a subset of a metric space $X$, 
\item[b.] $f:A\to \Delta(S)$ is a $(\delta,\delta)$-Lipschitz partition of unity on $A$ for some $\delta > 0$, 
\item[c.] $g:X\to \Delta(S)$ is a $(\delta,\delta)$-Lipschitz partition of unity on $X$, 
\item[d.] $p:X\to A$ is a retraction such that $d(x,p(x)) < dist(x,A)+1$ for all $x\in A$,
\item[e.] $\alpha:X\to [0,1]$ is $\frac{1}{r}$-Lipschitz, $\alpha(A)\subset \{0\}$, and $\alpha(X\setminus B(A,r))\subset \{1\}$,
\item[f.] $h:X\to \Delta(S)$ is defined as $h(x)=\alpha(x)\cdot g(x)+(1-\alpha(x))\cdot f(p(x))$.
\end{itemize}
In order for $h$ to be $(\epsilon,\epsilon)$-Lipschitz it suffices that $r\ge \frac{4}{\epsilon}$, $\delta\leq \frac{\epsilon}{3}-\frac{2}{3r}$, and $\delta\leq \frac{\epsilon}{4r+7}$.

If, in addition, the carriers of $f(A)$ and $g(X)$ are disjoint and both $f$ and $g$ are $M$-cobounded, then  $h$ is $(M+2r+2)$-cobounded.
\end{Lemma}
\begin{proof}
Notice $h$ is an extension of $f$.

We need to show $|h(x)-h(y)|\leq \epsilon\cdot d(x,y)+\epsilon$ for $x,y\in X$.
Notice
$h(x)-h(y)=\alpha(x)\cdot g(x)+(1-\alpha(x))\cdot f(p(x)) -[\alpha(y)\cdot g(y)+(1-\alpha(x))\cdot f(p(y))]=
(\alpha(x)-\alpha(y))\cdot g(x)+ \alpha(y)\cdot (g(x)-g(y))+[f(p(x))- f(p(y))]-[\alpha(x)\cdot f(p(x))-\alpha(y)\cdot f(p(y))]$.

The terms $(\alpha(x)-\alpha(y))\cdot g(x)$ and $\alpha(y)\cdot (g(x)-g(y))$ have universal estimates
$|(\alpha(x)-\alpha(y))\cdot g(x)|\leq |\alpha(x)-\alpha(y)|\leq \frac{1}{r}\cdot d(x,y)$
and $|\alpha(y)\cdot (g(x)-g(y))|\leq |g(x)-g(y)|\leq \delta\cdot d(x,y)+\delta$, so we need to estimate the remaining terms
depending of where $x$ and $y$ belong.

\textbf{Case 1}: $x\notin B(A,r)$ and $y\in B(A,r)$.\\
Here $\alpha(x)=1$, so $[f(p(x))- f(p(y))]-[\alpha(x)\cdot f(p(x))-\alpha(y)\cdot f(p(y))]=
(\alpha(y)-\alpha(x))\cdot f(p(y))$ and this term is at most $\frac{1}{r}\cdot d(x,y)$. Thus, in that case,
we have $|h(x)-h(y)|\leq (\frac{2}{r}+\delta)\cdot d(x,y)+\delta\leq \epsilon\cdot d(x,y)+\epsilon$.

\textbf{Case 2}: $x\in B(A,r)$ and $y\in B(A,r)$.\\
We know 
$|f(p(x))-f(p(y))|\leq \delta\cdot d(p(x),p(y))+\delta$. Notice $d(p(x),p(y))\leq d(p(x),x)+d(x,y)+d(y,p(y))\leq
dist(x,A)+1+d(x,y)+d(y,A)+1\leq 2r+2+d(x,y)$. Also, $\alpha(x)\cdot f(p(x))-\alpha(y)\cdot f(p(y))=
\alpha(x)\cdot (f(p(x))-f(p(y)))+(\alpha(x)-\alpha(y))\cdot f(p(y))$
resulting in $ |\alpha(x)\cdot f(p(x))-\alpha(y)\cdot f(p(y))|\leq |f(p(x))-f(p(y))|+|\alpha(x)-\alpha(y)|\leq
\delta\cdot (2r+2+d(x,y))+\delta+\frac{1}{r}\cdot d(x,y)$.

The final outcome is
$$|h(x)-h(y)|\leq $$
$$\frac{1}{r}\cdot d(x,y)+\delta\cdot d(x,y)+\delta+\delta\cdot (2r+2+d(x,y))+\delta+\delta\cdot (2r+2+d(x,y))+\delta+\frac{1}{r}\cdot d(x,y)=$$
$$(\frac{2}{r}+3\delta)\cdot d(x,y)+4r\delta+7\delta$$
To achieve $|h(x)-h(y)|\leq 
\epsilon\cdot d(x,y)+\epsilon$ it suffices $ \frac{2}{r}+3\delta\leq \epsilon$ and $ 4r\delta+7\delta\leq \epsilon$.
That amounts to $\delta\leq \frac{\epsilon}{3}-\frac{2}{3r}$ and $\delta\leq \frac{\epsilon}{4r+7}$.

\textbf{Case 3}: $x\notin B(A,r)$ and $y\notin B(A,r)$.\\
In that case $h(x)=g(x)$ and $h(y)=g(y)$, so $|h(x)-h(y)|\leq \delta\cdot d(x,y)+\delta\leq
\epsilon\cdot d(x,y)+\epsilon$.

Suppose the carriers of $f(A)$ and $g(X)$ are disjoint and there is $M > 0$ such that
$\diam(f^{-1}(st(v))), \diam(g^{-1}(st(v)))\leq M$ for all $v\in S$.

If $v\in S$ belongs to the carrier of $g(X)$ and $h(x)(v) > 0$, then $x$ must belong to $g^{-1}(st(v))$.
Thus, $\diam(h^{-1}(st(v)))\leq M$ in that case.
If $v\in S$ belongs to the carrier of $f(A)$ and $h(x)(v) > 0$, then $x\in B(A,r)$ and $p(x)\in f^{-1}(st(v))$.
Since $d(x,p(x))\leq r+1$, $\dist(x,f^{-1}(st(v))\leq r+1$ and  $\diam(h^{-1}(st(v)))\leq M+2r+2$.

\end{proof}

\section{Coarse normality}

In this section we dualize one part of Theorem \ref{Michael-NagamiThm}.

It is shown in \cite{DydExt} (see Theorem 9.1(5)) that a topological space $X$ is collectionwise normal if and only if partitions of unity on each closed subset $A$ of $X$ extends over $X$.
In other words, certain spaces are absolute extensors of $X$. \cite{DydakMitra lsAE} is devoted to dualizing the concept of absolute extensors to the coarse category. 

The following result may be seen as stating that every metric space $X$ is large scale
collectionwise normal.

\begin{Theorem}\label{NormalityTheorem}
For every $\epsilon > 0$ there is $\delta > 0$ such that any $(\delta,\delta)$-Lipschitz partition of unity $f:A\to \Delta(S)$, $A$ a subset of a metric space $X$,
extends to an $(\epsilon,\epsilon)$-Lipschitz partition of unity $g:X\to \Delta(S)$.
\end{Theorem}
\begin{proof}
Pick $r= \frac{8}{\epsilon}$. Once $r$ is fixed choose $\delta$ smaller than both
$\frac{\epsilon}{3}-\frac{2}{3r}=\frac{\epsilon}{4}$ and $\frac{\epsilon}{4r+7}$. 
Suppose $f:A\to \Delta(S)$ is a $(\delta,\delta)$-Lipschitz partition of unity on $A$.
Obviously, there is a retraction $p:X\to A$ such that $d(x,p(x)) < dist(x,A)+1$ for all $x\in A$.
Consider $\alpha:X\to [0,1]$ defined by $\alpha(x)=\min(\frac{d(x,A)}{r},1)$. Notice it is $\frac{1}{r}$-Lipschitz.
Define $g:X\to \Delta(S)$ via $g(x)=\alpha(x)\cdot v+(1-\alpha(x))\cdot f(p(x))$, where $v$ is some fixed point in $S$.
By \ref{BasicExtLemma}, $g$
 extends $f$ and is $(\epsilon,\epsilon)$-Lipschitz.
\end{proof}

\section{Unifying asymptotic dimension and large scale paracompactness}

In this section we develop a result that allows a unified approach to both asymptotic dimension and large scale paracompactness.

Classical dimension theory of topological spaces has the following three threads that are relevant to this paper (the fourth thread is that of inductive definitions of dimension):
\begin{itemize}
\item dimension defined using multiplicity of covers (commonly known as the covering dimension),
\item Ostrand-Kolmogorov version of covering dimension (see \cite{O$_1$} and
\cite{OstrandDimofMetriSpacesHilbert}),
\item dimension defined via extending maps to spheres. 
\end{itemize}
 
 Gromov \cite{Grom} defined asymptotic dimension by interpreting the first thread. It turns out that definition also generalizes the second thread as seen in Theorem 9.9 (p.131 of \cite{Roe lectures}). The definition of asymptotic dimension in \cite{Roe lectures} (see p.129) can be translated using \cite{DyHo} to the language of uniformly bounded covers (as opposed to the language of controlled sets of \cite{Roe lectures}) as follows:
 
 \begin{Definition}\label{AsDimDef}
A coarse space $X$ has \textbf{asymptotic dimension} at most $n$ ($n$ a given non-negative integer)
if for every uniformly bounded cover $\mathcal{U}$ of $X$ there exist uniformly bounded families $\mathcal{V}_0,\ldots,\mathcal{V}_n$
that are $\mathcal{U}$-disjoint (i.e. each element of $\mathcal{U}$ intersects at most one element of $\mathcal{V}_i$)
and $X=\bigcup\limits_{i=0}^n \mathcal{V}_i$.
\end{Definition}

Definition \ref{AsDimDef} is in the spirit of Ostrand-Kolmogorov and is equivalent to the following (see Theorem 9.9 on p.131 in \cite{Roe lectures}): 
A coarse space $X$ has \textbf{asymptotic dimension} at most $n$ (notation: $\asdim(X)\leq n$, $n$ a given non-negative integer)
if for every uniformly bounded cover $\mathcal{U}$ of $X$ there exists a uniformly bounded cover $\mathcal{V}$ of $X$ such that  each element of $\mathcal{U}$ intersects at most $n+1$ elements of $\mathcal{V}$.

The first attempt to generalize the third thread of dimension theory was initiated by Dranishnikov \cite{Dran AsyTop}.
\cite{DydakMitra lsAE} contains a different take on that issue and it centers on the concept of a large scale absolute extensor.
Recall that, in case $K$ is a bounded metric space, $K$ is a \textbf{large scale absolute extensor} of $X$ if for all $\epsilon > 0$ there is $\delta > 0$ such that for any subset $A$ of
$X$ any $(\delta,\delta)$-Lipschitz function $f\colon A\to K$
extends to an $(\epsilon,\epsilon)$-Lipschitz function $g\colon X\to K$ (see \cite{DydakMitra lsAE}). 

It turns out (see \cite{DydakMitra lsAE}) that $S^n$ being a large scale absolute extensor of $X$ is related to the dimension of the Higson corona of $X$ being at most $n$ (in case $X$ is a proper metric space) and, if $X$ is of finite asymptotic dimension, then
it is equivalent to $\asdim(X)\leq n$. It remains an open problem if $\asdim(X)\leq n$ provided $S^n$ is a large scale absolute extensor of $X$. In this section we propose another version of generalizing the third thread of dimension theory as follows:

\begin{Definition}\label{AsdimRelFunction}
 Let $X$ be a metric space, $n\leq \infty$, $\alpha$ be a function on a subset $D_{\alpha}$ of $(0,\infty)$ to $(0,\infty)$, and $M: D_{\alpha}\times (0,\infty)\to (0,\infty)$ be a function. We say the \textbf{large scale extension dimension of} $X$ \textbf{with respect to} $\alpha$ and $M$ is at most $n$
(notation $\lsdim(X,\alpha,M)\leq n$) if for any set $S$ of cardinality bigger than $\card(X\times \mathbb{N})$, any $K > 0$, any $(\alpha(\delta),\alpha(\delta))$-Lipschitz map $f:A\subset X\to \Delta(S)^{(n)}$ ($\delta\in D_{\alpha}$) that is $K$-cobounded
extends to a $(\delta,\delta)$-Lipschitz map $g:X\to \Delta(S)^{(n)}$ that is $M(\delta,K)$-cobounded.
\end{Definition}

\begin{Remark}
Notice that if Definition \ref{AsdimRelFunction} holds for one set $S$, then it holds for any set of cardinality bigger than $\card(X\times \mathbb{N})$. Indeed, given a partition of unity $f:A\subset X\to \Delta(S)$, the carrier of $f$ has vertices forming a set of cardinality at most $\card(X\times \mathbb{N})$. That can be easily established by noticing that, for each $k\ge 0$, vertices generated by $x\in A$ such that $f(x)$ is in the geometric interior of a $k$-simplex, form a set of cardinality at most $\card(X\times \mathbb{N})$.
\end{Remark}

\begin{Theorem}\label{MainGeneralizationThirdThread}
 Let $X$ be a metric space, $n\leq \infty$, and $S$ is a set of cardinality bigger that $\card(X\times \mathbb{N})$. The following conditions are equivalent:
\begin{itemize}
\item[1.] For each $\epsilon > 0$ there is an $(\epsilon,\epsilon)$-Lipschitz partition of unity
$f:X\to \Delta(S)^{(n)}$ such that the family $\{f^{-1}(st(v))\}_{v\in S}$ is uniformly bounded.

\item[2.] There are functions $\alpha: (0,\infty)\to (0,\infty)$, $M: (0,\infty)\times (0,\infty)\to (0,\infty)$ such that $\lsdim(X,\alpha, M)\leq n$.
\end{itemize}
\end{Theorem}
\begin{proof}
2)$\implies$1). Let $A$ be a point in $X$ and let $f:A\to \Delta(S)^{(n)}$ be a constant map
to a vertex.
For each $\epsilon > 0$, $f$ is $(\alpha(\epsilon),\alpha(\epsilon))$-Lipschitz and $1$-cobounded, so it extends to an $(\epsilon,\epsilon)$-Lipschitz $g:X\to \Delta(S)^{(n)}$
that is $M(\epsilon,1)$-cobounded.
\par
1)$\implies$2). 
Suppose $\epsilon > 0$ and $K > 0$.
Pick $\mu > 0$ with the property that for any $(\mu,\mu)$-Lipschitz partition of unity
$g:X\to \Delta(S)$ there is an $(\epsilon,\epsilon)$-Lipschitz 
$h:X\to \Delta(S)^{(n)}$ so that $g(x)\in \Delta(S)^{(n)}$ implies $h(x)=g(x)$
and $h(x)(v) > 0$ implies $g(x)(v) > 0$ for all $x\in X$ and $v\in S$.
For $n < \infty$ existence of $\mu$ is established in \cite{CDV2}, for $n=\infty$ we put $\mu=\epsilon$ (as $h=g$ works).

Pick $r= \frac{8}{\mu}$. Once $r$ is fixed choose $\delta$ smaller than both
$\frac{\mu}{3}-\frac{2}{3r}=\frac{\mu}{4}$ and $\frac{\mu}{4r+7}$. Put $\alpha(\epsilon)=\delta$.
Suppose $f:A\to \Delta(S)$ is a $(\delta,\delta)$-Lipschitz partition of unity on $A$
that is $K$-cobounded.
Obviously, there is a retraction $p:X\to A$ such that $d(x,p(x)) < dist(x,A)+1$ for all $x\in A$.
Consider $\gamma:X\to [0,1]$ defined by $\gamma(x)=\min(\frac{d(x,A)}{r},1)$. Notice it is $\frac{1}{r}$-Lipschitz.
Define $g:X\to \Delta(S)$ via $g(x)=\gamma(x)\cdot u(x)+(1-\alpha(x))\cdot f(p(x))$, where $u$ is some $(\delta,\delta)$-Lipschitz partition of unity $u:X\to \Delta(S)^{(n)}$ that is $Q$-cobounded for some $Q > 0$.
By \ref{BasicExtLemma} $g$
 extends $f$, is $(\mu,\mu)$-Lipschitz, and is $(\max(K,Q)+2r+2)$-cobounded.
Now, modify $g$ to obtain an $(\epsilon,\epsilon)$-Lipschitz 
$h:X\to \Delta(S)^{(n)}$ so that $g(x)\in \Delta(S)^{(n)}$ implies $h(x)=g(x)$
and $h(x)(v) > 0$ implies $g(x)(v) > 0$ for all $x\in X$ and $v\in S$.
Notice $h$ is $\max(K,Q)+2r+2$-cobounded. That means putting $M(\epsilon,K)= \max(K,Q)+2r+2$ works and the proof is completed.
\end{proof}

\begin{Remark}
Notice Theorem \ref{MainGeneralizationThirdThread} provides a very good unification of Property A and asymptotic dimension.
For $n$ finite, Condition 1 in \ref{MainGeneralizationThirdThread} amounts to $\asdim(X)\leq n$.
For $n=\infty$ that condition is equivalent to $X$ being large scale paracompact which, in case of $X$ being of bounded geometry, is equivalent to $X$ having Property A (see  \cite{CDV4}).
\end{Remark}

 \section{Countable asymptotic dimension}
 This section is devoted to generalizing Definition \ref{AsDimDef} to the case of infinite asymptotic dimension.
 Using the Ostrand-Kolmogorov approach as a blueprint (and in analogy to the concept of countable covering dimension)
 we propose the following:
 
\begin{Definition}\label{CountableAsdim}
A metric space $X$ is of \textbf{countable asymptotic dimension} if there is a sequence of integers $n_i\ge 1$, $i\ge 1$, such that for any sequence of positive real numbers $R_i$, $i\ge 1$, there is a sequence
$\mathcal{V}_i$ of families of subsets of $X$ such that the following conditions are satisfied:
\begin{itemize}
\item[1.] $\mathcal{V}_1=\{X\}$,
\item[2.] each element $U\in \mathcal{V}_i$ can be expressed as a union of at most $n_i$ families from $\mathcal{V}_{i+1}$ that are $R_i$-disjoint,
\item[3.] at least one of the families $\mathcal{V}_i$ is uniformly bounded.
\end{itemize}
\end{Definition}

\begin{Proposition}
If a metric space $X$ is of straight finite decomposition complexity, then $X$ is of countable asymptotic dimension.
\end{Proposition}
\begin{proof}
Recall $X$ is of \textbf{straight finite decomposition complexity} \cite{DranZari}  if 
for any increasing sequence of positive real numbers $R_1 < R_2 < \ldots$ there a sequence
$\mathcal{V}_i$, $i \leq n$, of families of subsets of $X$ such that the following conditions are satisfied:
\begin{itemize}
\item[1.] $\mathcal{V}_1=\{X\}$,
\item[2.] each element $U\in \mathcal{V}_i$, $i < n$, can be expressed as a union of at most $2$ families from $\mathcal{V}_{i+1}$ that are $R_i$-disjoint,
\item[3.] $\mathcal{V}_n$ is uniformly bounded.
\end{itemize}
That means $n_i=2$ for $i\ge 1$ works.
\end{proof}


Our next concept generalizes Definition \ref{AsdimRelFunction}.

\begin{Definition}\label{AsdimRelFunction2}
Suppose $X$ is a subset of a metric space $Y$, $n\leq \infty$, $\alpha$ is a function on a subset $D_{\alpha}$ of $(0,\infty)$ to $(0,\infty)$, and $M: D_{\alpha}\times (0,\infty)\to (0,\infty)$ is a function. We say the \textbf{large scale extension dimension of} $X$ \textbf{with respect to} $Y$, $\alpha$, and $M$ is at most $n$
(notation $\lsdim(X,Y,\alpha,M)\leq n$) if for any set $S$ of cardinality bigger than $\card(Y\times \mathbb{N})$, any $K > 0$, any $(\alpha(\delta),\alpha(\delta))$-Lipschitz map $f:A\subset Y\to \Delta(S)^{(n)}$ ($\delta\in D_{\alpha}$) that is $K$-cobounded
extends to a $(\delta,\delta)$-Lipschitz map $g:A\cup X\to \Delta(S)^{(n)}$ that is $M(\epsilon,K)$-cobounded.
\end{Definition}

\begin{Lemma}\label{BasicAsdimLemma}
Suppose $\alpha:[a,\infty)\to [b,\infty)$ and $\beta:[b,\infty)\to (0,\infty)$ are functions.
Let $\{W_t\}_{t\in T}$ be an $R$-disjoint family of subsets of $X$ such that
 $\lsdim(W_t,X,\alpha,M)\leq n$ for each $t\in T$. If $\lsdim(B,X,\beta,M_B)\leq n$ for some $B\subset X$,
then $$\lsdim(B\cup\bigcup_{t\in T} W_t,X,\beta\circ\alpha,M_1)\leq n$$
provided $a\ge \frac{2}{R+1}$ and $M_1(u,K)=2\cdot M(u,M_B(\alpha(u),K))+M_B(\alpha(u),K)$.
\end{Lemma}
\begin{proof}
Suppose $A\subset X$ and $f:A\to \Delta(S)^{n}$ is $(\beta\circ\alpha(u),\beta\circ\alpha(u))$-Lipschitz
and $K$-cobounded for some $u\ge a$. Extend it to $g:A\cup B\to \Delta(S)^{n}$ which is $(\alpha(u),\alpha(u)$-Lipschitz
and $M_B(\alpha(u),K)$-cobounded. Now, for any $t\in T$, $g$ extends over $W_t$ to a $g_t$ function that
is $(u,u)$-Lipschitz and $M(u,M_B(\alpha(u),K))$-cobounded. We may arrange so that for $t_1\ne t_2$
new vertices introduced during extension are different.
Since $u\ge \frac{2}{R+1}$, $h=f\cup \bigcup_{t\in T}g_t$ is $(u,u)$-Lipschitz by \ref{BasicLipLemma}.
$h$ is $(2\cdot M(u,M_B(\alpha(u),K))+M_B(\alpha(u),K))$-cobounded. Indeed, new vertices have point inverses of their stars
arising from a single map $g_t$, so they are bounded by $M(u,M_B(\alpha(u),K))$.
Old vertices $v$ have their main part $g^{-1}(st(v))\ne\emptyset$ (of diameter at most $M_B(\alpha(u),K)$) enlarged by adding $g_t^{-1}(st(v))$ for each $t\in T$. Each union $g^{-1}(st(v))\cup g_t^{-1}(st(v))$ is of diameter at most $M(u,M_B(\alpha(u),K))$ resulting in $h$ being $M_1(u,K)$-cobounded. 

\end{proof}

\begin{Lemma}\label{BasicAsdimLemma2}
Suppose $\alpha:(0,\infty)\to (0,\infty)$ is a non-decreasing function such that 
the $q$-fold composition $\alpha^q$ satisfies $\alpha^q(a)\ge \frac{2}{R+1}$ for some $R > 0$, $q\ge 1$, and all $a > 0$. 
Let $M:[\alpha^q(a),\infty)\times (0,\infty)\to (0,\infty)$ be a function and
consider the family $\mathcal{V}$ of all subsets $W$ of $X$ satisfying
 $\lsdim(W,X,\alpha|[\alpha^q(a),\infty),M)\leq n$. There is a function $M_1:[a,\infty)\times (0,\infty)\to (0,\infty)$  such that if $U\subset X$ is the union of $q$ families in $\mathcal{V}$
that are $R$-disjoint, then $\lsdim(U,X,\alpha^q|[a,\infty),M_1)\leq n$.
\end{Lemma}
\begin{proof}
For $q=1$ it follows from Lemma \ref{BasicAsdimLemma}. Use induction on $q$ and apply Lemma \ref{BasicAsdimLemma} again as follows. Suppose $B\subset X$ is the union of $q-1$ families in $\mathcal{V}$
that are $R$-disjoint. By inductive assumption (we use $\alpha(a)$ instead of $a$), $\lsdim(B,X,\alpha^{q-1}|[\alpha(a),\infty),M_2)\leq n$ for some function $M_2:[\alpha(a),\infty)\times (0,\infty)\to (0,\infty)$. If $W$ is the union of a family in $\mathcal{V}$
that is $R$-disjoint, then put $\beta=\alpha^{q-1}|[\alpha(a),\infty)$.
Notice $\beta\circ\alpha=\alpha^q|[a,\infty)$.
Using \ref{BasicAsdimLemma}, we get 
$$\lsdim(B\cup W,\alpha^q|[a,\infty), M_1)\leq n$$
for $M_1:[a,\infty)\times (0,\infty)\to (0,\infty)$ defined by
$M_1(u,K)= 2\cdot M(u,M_2(\alpha(u),K))+M_2(\alpha(u),K)$.
\end{proof}

\begin{Theorem}\label{MainResult}
 Let $X$ be a metric space and $n\leq \infty$ such that $\Delta(X)^{(n)}$ is a large scale absolute extensor of $X$. If $X$ is of countable asymptotic dimension, then $\lsdim(X)\leq n$.
\end{Theorem}
\begin{proof}

Pick a function $E:(0,\infty)\to (0,\infty)$ such that $E(x) < x$ for all $x$ and 
any $(E(x),E(x))$-Lipschitz function $f:A\subset X\to \Delta(X)^{(n)}$ extends to an $(x,x)$-Lipschitz function $g:X\to \Delta(X)^{(n)}$. We may assume $E$ is non-decreasing (replace $E(x)$ by
$\sup\{E(t)/2 | t < x\}$ if necessary).
Suppose $S$ is a set of cardinality bigger than $\card(X\times \mathbb{N})$.
Point out that any $(E(x),E(x))$-Lipschitz function $f:A\subset X\to \Delta(S)^{(n)}$ extends to an $(x,x)$-Lipschitz function $g:X\to \Delta(S)^{(n)}$.
Given $k > 0$, by $E^k$ we mean the composition $E\circ\ldots\circ E$ of $k$ copies of $E$.
$E^0=id$.

There is a sequence of integers $n_i\ge 1$, $i\ge 1$, such that for any sequence of positive real numbers $R_i$, $i\ge 1$, there is a sequence
$\mathcal{V}_i$ of families of subsets of $X$ such that the following conditions are satisfied:
\begin{itemize}
\item[1.] $\mathcal{V}_1=\{X\}$,
\item[2.] each element $U\in \mathcal{V}_i$ can be expressed as a union of at most $n_i$ families from $\mathcal{V}_{i+1}$ that are $R_i$-disjoint,
\item[3.] at least one of the families $\mathcal{V}_i$ is uniformly bounded.
\end{itemize}

Let $N(1)=0$ and let $N(i)=\prod\limits_{j=1}^{i-1} n_j$ for $i\ge 1$.

Given $2 > \epsilon > 0$ define $R_i > 0$ as satisfying $\frac{2}{R_i+1} = E^{N(i)}(\epsilon)$, then pick a sequence
$\mathcal{V}_i$ of families of subsets of $X$ satisfying the above conditions. 
 Choose $m \ge 1$ such that $\mathcal{V}_m$ is uniformly bounded by $K$.

\textbf{Claim 1}: 
$\lsdim(U,X,E|[\epsilon,\infty),M_m)\leq n$ for all $U\in \mathcal{V}_m$, where $M_m:[\epsilon,\infty)\times (0,\infty)\to (0,\infty)$
defined by $M_m(x,y)=y+K+R_m$.\\
\textbf{Proof of Claim 1}:
Suppose $u\ge \epsilon$, $f:A\subset X \to \Delta(S)^{(n)}$ is $(E(u),E(u))$-Lipschitz and $R$-cobounded.
If $A\cap B(U,R_m)=\emptyset$, then extending $f$ to $g:A\cup U\to \Delta(S)^{(n)}$
by sending $U$ to a vertex $v_U$ not belonging to the carrier of $f(A)$ produces a $(u,u)$-Lipschitz function by \ref{BasicLipLemma}
that is $(R+K)$-cobounded. Indeed, $g^{-1}(st(v_U))=U$ is of diameter at most $K$
and $g^{-1}(st(v))=f^{-1}(st(v))$ for $v\ne v_U$ is of diameter at most $R$.
\par
Extend $f$ to $g:A\cup U\to \Delta(S)^{(n)}$ that is $(u,u)$-Lipschitz.
This may give rise to points $x\in U$ and $a\in A$ that are far away but
both $g(a)$ and $g(x)$ belong to the same star. To avoid that difficulty, consider
the vertices $S_1$ of the carrier of $f(A\cap B(U,R_m))$ and the vertices
$S_2\supset S_1$ of the carrier of $g(B(U,R_m))$. Let $r:S_2\to S_1$ be a retraction.
Change $g$ to $h$ by changing it on $B(U,R_m)$ to the composition of
$g$ and the induced retraction $\Delta(S_2)\to\Delta(S_1)$. $h$ is $(u,u)$-Lipschitz (see \ref{BasicLipLemma}), it extends $f$, and to check it is
$(R+K+R_m)$-cobounded all one has to do is look at $h^{-1}(st(v))$ for $v\in S_1$.
This set contains $a\in A\cap B(U,R_m)$, its intersection with $A$ is of diameter at most $R$, and the remainder is contained in $U$. Therefore any two points of $h^{-1}(st(v))$ are at the distance at most $R+R_m+K$.
That completes the proof of Claim 1.

Define $P(m)=1$ and $P(i)=P(i+1)\cdot n_i$ for $i < m$.

\textbf{Claim 2}: 
For each $1\leq i\leq m$ there is a function $M_{i}:[E^{N(i)}(\epsilon),\infty)\times (0,\infty)\to (0,\infty)$ such that
$\lsdim(U,X,E^{P(i)}|[E^{N(i)}(\epsilon),\infty),M_i)\leq n$ for all $U\in \mathcal{V}_i$.\\
\textbf{Proof of Claim 2}: $i=m$ is taken care of by by Claim 1.
Suppose $i < m$ and $M_{i+1}$ exists.
Put $q=n(i)$ and $\alpha=E^{P(i+1)}:[E^{N(i+1)}(\epsilon),\infty)\to (0,\infty)$. Applying Lemma \ref{BasicAsdimLemma2}
one gets the existence of a function $M_{i}:[E^{N(i)}(\epsilon),\infty)\times (0,\infty)\to (0,\infty)$
such that $\lsdim(U,X,E^{P(i)}|[E^{N(i)}(\epsilon),\infty),M_{i})\leq n$ for all $U\in \mathcal{V}_{i}$.

Applying Claim 2 to $i=1$ we get $\lsdim(X,X,E^{P(1)}|[\epsilon,\infty),M_1)\leq n$.
That implies existence of an $(\epsilon,\epsilon)$-Lipschitz function $g:X\to \Delta(S)^{(n)}$
that is $K$-cobounded for some $K > 0$. Thus, $\lsdim(X)\leq n$.
\end{proof}

Now we can derive a more general result than Theorem \ref{DranZariResult}.
\begin{Corollary}\label{CountableAsDimImpliesPropA}
Any space $X$ of countable asymptotic dimension has Property A.
\end{Corollary}
\begin{proof}
We are applying Theorem \ref{MainResult} when $n=\infty$
in which case the assumption $\Delta(X)^{(n)}$ is a large scale absolute extensor of $X$ is vacuous (in view of Theorem \ref{NormalityTheorem}).

Notice $X$ is large scale finitistic (see \ref{LSFinitisticDef}), hence it is large scale weakly paracompact. In view of Theorem \ref{MainResult} for each $\epsilon > 0$
there is an $(\epsilon,\epsilon)$-Lipschitz partition of unity on $X$ that is cobounded.
As shown in \cite{CDV4} (use Theorem 4.9 there which says that if $X$ is large scale weakly paracompact
and for each $\epsilon > 0$
there is an $(\epsilon,\epsilon)$-Lipschitz partition of unity on $X$ that is cobounded,
then $X$ is large scale paracompact),
a large scale finitistic metric space $X$ has Property A if and only if it is large scale paracompact. Consequently, $X$ has Property A.
\end{proof}

\begin{Remark}
Theorem \ref{MainResult} is related to the problem of A.Dranishnikov about the equality of asymptotic dimension $\asdim(X)$ of proper metric spaces $X$ to the covering dimension of their Higson corona $\nu(X)$ (see \cite{Dran AsyTop}). As is shown in \cite{Dran AsyTop}  and \cite{DranKeesUsp} the two numbers are equal in case of $\asdim(X)$ being finite. Theorem \ref{MainResult} improves that result for spaces of countable asymptotic dimension. Note (see \cite{DydakMitra lsAE}) that $\dim(\nu(X))\leq n$ is equivalent to the $n$-sphere $S^n$ being a large scale absolute extensor of $X$.
\end{Remark}

\begin{Remark}
In a recent paper \cite{RR}, D. A. Ramras and B. W. Ramsey introduced independently the concept
of a metric family $\X$ to have \emph{weak straight finite decomposition complexity}
with respect to the sequence 
$(k_1, k_2, \ldots)$ $($$k_i\in \N$$)$ if
 for every sequence
$R_1 < R_2 < \ldots$ of positive numbers, there exists an $n \in \N$  and metric families $\X_0, \X_1, \X_2, \ldots, \X_n$
such that $\X = \X_0$, the family $\X_{i}$ is $(k_{i+1}, R_{i+1})$--decomposable over $\X_{i+1}$, and the
family $\X_n$ is uniformly bounded. $\X$ has weak  straight finite decomposition complexity (wsFDC)
if it has wsFDC with respect to some sequence $(k_1, k_2, \ldots)$.

Notice that, in case of $\X$ consisting of a single space $X$, the above definition amounts to saying that $X$ has countable asymptotic dimension. Therefore our Corollary \ref{CountableAsDimImpliesPropA}
answers positively Question 4.7 of \cite{RR}.
\end{Remark}

\end{document}